\documentclass[oneside]{amsart}
\usepackage{amsfonts}
\usepackage{amssymb}
\usepackage{amsxtra}
\usepackage{amstext}
\usepackage{url}
\usepackage[english]{babel}

\newcommand{\R}{\mathbb{R}}

\newcommand{\Z}{\mathbb{Z}}

\newcommand{\pp}{\mathbb{P}}

\newcommand{\kO}{\mathcal{O}}

\newcommand{\kF}{\mathcal{F}}

\newtheorem {lem} {Lemma} %[section]
\newtheorem {prop} {Proposition} %[section]

\newtheorem {theo*} {Theorem}

\newtheorem* {rem} {Remark} %[section]

\ifx\hyperlink\undefined
\newcommand{\hurl}[2]{\url{#1}}
\else
\newcommand{\hurl}[2]{\href{#2}{#1}}
\fi

\title%[A Martingale Excited Random Walk]
      {A balanced excited random walk}

\author{Ita{\i} Benjamini}
\author{Gady Kozma}
\address{The Weizmann Institute of Science, Rehovot POB 76100 Israel}
\email{itai.benjamini@weizmann.ac.il}
\email{gady.kozma@weizmann.ac.il}

\author{Bruno SCHAPIRA}
\address{D\'epartement de Math\'ematiques, B\^at. 425, Universit\'e Paris-Sud 11, F-91405 Orsay, cedex, France. }
\email{bruno.schapira@math.u-psud.fr}

\begin{document}

\begin{abstract}
The following random process on $\Z^4$ is studied. At first visit to a site, the two first coordinates perform a ($2$-dimensional) simple random walk step. At further visits, it is the last two coordinates which perform a simple random walk step. We prove that this process is almost surely transient. The lower dimensional versions are discussed and various generalizations and related questions are proposed.
\end{abstract}

\maketitle

\section{Introduction}
%In this note we present a new model of excited random walk on $\Z^d$, $d\ge 2$, where the different laws used at each step are all symmetric and to the nearest neighbor. This is in contrast with the original model of excited random walks introduced by Benjamini and Wilson \cite{BW}, where the law used at fresh sites had a bias in some fixed direction. One  motivation is to come closer to the standard models of reinforced random walks on $\Z^d$, which are symmetric in nature. We think for instance on the question of recurrence vs transience of $1$-reinforced random walks, which is still
%widely open (see the surveys \cite{MR} and \cite{Pem}). Note that a model of excited random walks on $\Z$ with only symmetric laws had already been considered in \cite{KRS}, but there the laws were not to the nearest neighbors. See \cite{G} for other variants.
Excited random walk as defined by Benjamini and Wilson \cite{BW} has
a bias in some fixed direction, a feature which is highly useful in its
analysis. See e.g.\ \cite{MPRV} and references within. Attempts to relax
the dependence of the proof structure on monotonicity resulted in a
number of works where the walker has competing drifts. See \cite{ABK,
  KZ, H}. One motivation was to get closer to standard models of
reinforced random walks on $\Z^d$, which are symmetric in nature. We
think for instance on the question of recurrence vs.\ transience of
$1$-reinforced random walks, which is still widely open (see the
surveys \cite{MR} and \cite{Pem}). With the same goal in mind, we
started exploring excited-like models where the walker is in addition
also a martingale or a bounded perturbation of one, and posed some
questions in 2007 \cite{G}, which, it seems, are all still
open. Progress on this kind of models was achieved in \cite{KRS}, but
there the laws were not nearest-neighbors. Here we describe and solve
one such model of a nearest-neighbor walk in 4 dimensions.

We describe a general form of the model in any dimension $d\ge 2$, but
we will actually only deal with $4$ dimensions here.
So one has first to choose arbitrarily two integers $d_1\ge 1$ and $d_2\ge 1$ such that $d=d_1+d_2$.
Then we define the process $(S_n,n\ge 0)$ on $\Z^d$ as a mixture of two simple random walks in the following sense.
Set $S_n=(X_n,Y_n)$, where $X_n \in \Z^{d_1}$ is the set of the first $d_1$ coordinates of $S_n$ and $Y_n\in \Z^{d_2}$ is the set of the last $d_2$ coordinates. Now the rule is the following. First $S_0=0$. Next if $S$ visits a site for the first time then only the $X$ component performs a simple random walk step, that is:
$$\pp[S_{n+1}-S_n = (0,\dots,0,\pm 1,0,\dots,0)\mid \kF_n] = 1/(2d_1),$$
where the $\pm 1$ can be at any of the first $d_1$ coordinates. Otherwise, only $Y$ performs a simple random walk step:
$$\pp[S_{n+1}-S_n = (0,\dots,0,\pm 1,0,\dots,0)\mid \kF_n] = 1/(2d_2),$$
if $S$ already visited the site $S_n$ in the past, where this time the $\pm 1$ can be at any of the last $d_2$ coordinates.
We call this process $S$ the $M(d_1,d_2)$-random walk.

Here we say that a process is transient if almost surely any site is visited only finitely many times. It is said to be recurrent if almost surely it visits all sites infinitely often.
We will prove the following,
\begin{theo*}
\label{theo}
The $M(2,2)$-random walk is transient.
\end{theo*}
The proof of Theorem \ref{theo} is elementary, and uses only basic estimates on the standard $2$-dimensional simple random walk.
It relies on finding good upper bounds for the probability of return to the origin and then use Borel-Cantelli Lemma
(what makes however the proof nontrivial is that the two components $X$ and $Y$ are not independent).

\vspace{0.2cm}
Note that the canonical projections of $S$ on $\Z^{d_1}$ and $\Z^{d_2}$ are usual (time changed) simple random walks. So if $d=4$, and if $d_1$ or $d_2$ equals $3$, then $S$ is automatically transient, since the simple random walk on $\Z^3$ is transient. Likewise if $d$ is larger than $5$, then for any choice of $d_1$ and $d_2$ (larger than $1$), the resulting process will be transient. Thus the question of recurrence vs transience is only interesting in dimension less than $4$. In dimension $3$ there are two versions: $d_1=1$ and $d_2=2$ or $d_1=2$ and $d_2=1$.
We conjecture that in both cases $S$ will be transient, also because it is a $3$-dimensional process. Proving this seems nontrivial, but notice that a possible intermediate step between the dimension $3$ and $4$ could be to consider the analogue problem on the discrete $3$-dimensional Heisenberg group, which is generated by $2$ elements (and their inverses), yet balls of radius $r$ has size order $r^4$.  

\vspace{0.2cm}
Let us make some comments now on the $2$-dimensional case. As a
$2$-dimensional process, we believe that $M(1,1)$ is
recurrent. Observe however that this is not true when starting from
any configuration of visited sites. Indeed if we start with a vertical
line of visited sites, then the process will be trapped in this line,
and if the line does not include the origin, the process will not
return there. It is also not difficult to construct starting
environments such that the first coordinate of the process will tend
almost surely toward $+\infty$. For example if the initial
configuration is the ``trumpet'' $\{(x,y):|y|<e^x\}$ then the walker will
drift to infinity in the $x$ direction\footnote{We will not prove any
  of these claims, as they are somewhat off-topic}. Of course it is not possible for the random walk to create these environments in finite time, so it is not an obstacle for recurrence, but it may be interesting to keep this in mind. Another problem concerns the limiting shape of the range (i.e.~the set of visited sites) of the process. Based on heuristics and some simulations, we believe that it is a vertical interval. This problem is closely related to the question of evaluating the size of the range $R_n$ at time $n$. Indeed the horizontal displacement of the process at time $n$ is of order $\sqrt{\# R_n}$, whereas its vertical displacement is always of order $\sqrt{n}$. So another formulation of the problem would be to show that $R_n$ is sublinear. By the way we mention a related question. Assume that at each step, one can decide, conditionally on the past, to move the first coordinate or the second coordinate (and then perform a $1$-dimensional simple random walk step). Then what is the best strategy to maximize the range? In particular is it possible for the range to  be of size roughly $n$, or at least significantly larger than $n/\ln n$, which is the size of the range of the simple random walk?

\vspace{0.2cm}
A possible generalization of our model would be to consider multi-excited versions, in the spirit of Zerner \cite{Z}.
In this case one should first decompose $d$ as $d=d_1+\dots+d_m$, for some $m\ge 2$ and $d_i\ge 1$, with $i\le m$.
Then at $i^\textrm{th}$ visit to a site only the $i^\textrm{th}$ component of $S$ performs a simple random walk step, if $i<m$, and at further visits
only the $m^\textrm{th}$ component moves. In dimension $4$ for instance the case $d_1=2$ and $d_2=d_3=1$ seems interesting and nontrivial. Another interesting case is
$d\ge 3$ and $d_i=1$ for each $i\le d$ (even the case $d$ very large seems nontrivial).

\vspace{0.2cm}
Another related problem is the following. Take two symmetric laws $\mu_1$ and $\mu_2$ on $\Z^4$. Decide that at first visit to a site the jump of the process has law $\mu_1$, and at further visits it has law $\mu_2$. Then is it true that if the support of $\mu_1$ and $\mu_2$ both generate $\Z^4$, then the process is transient?

\vspace{0.2cm}
\noindent \textit{Acknowledgments:} This work was done while BS was a visitor at the Weizmann Institute and he thanks this institution
for its kind hospitality.

\section{Proof of the theorem}

\noindent The theorem is a direct consequence of the following proposition.

\begin{prop}
\label{prop}
There exists a constant $C>0$ such that for any $n>1$,
$$\pp\left[0\in \{S_n,\dots,S_{2n}\}\right] \le C\left(\frac{\ln \ln n}{\ln n}\right)^2.$$
\end{prop}

\noindent Indeed assuming this proposition we get
$$\sum_{k\ge 0} \pp\left[0\in \{S_{2^k},\dots,S_{2^{k+1}}\}\right]<+\infty,$$
and we can conclude by using the Borel-Cantelli lemma. So all we have to do is to prove this proposition.

\vspace{0.2cm}
\noindent \textit{Proof of Proposition \ref{prop}.} For any $n\ge 1$, denote by $r_n$ the cardinality of the range of $S$ at time $n$. The next lemma will be needed:

\begin{lem}
\label{lem}
For any $M>0$, there exists a constant $C>0$, such that
$$\pp\left[n/(C\ln n)^2\le r_n\le 99n/100\right] = 1-o(n^{-M}).$$
\end{lem}
\begin{proof} Note first that for any $k$, if $S_k$ and $S_{k+1}$ were not already visited in the past, then
$S_{k+2}= S_k$ with probability at least $1/4$.
In particular for any $k$, there is probability at least $1/4$ that $S$ is not at a fresh site at one of the time $k$, $k+1$ or $k+2$.
Then a standard use of the Azuma-Hoeffding inequality gives the desired upper bound on $r_n$.

We now prove the lower bound. Let $c>0$ be fixed. Let $(U_n,n\ge 0)$ be a simple random walk on $\Z^2$.
For any $n\ge 1$ and $x\in \Z^2$, denote by $N_n(x)$ the number of
visits of $U$ to $x$ before time $n$. A simple and standard
calculation (see e.g.\ \cite[Proposition 4.2.4]{LL}) shows that there
exists a constant $C>0$ such that the probability to not visit $x$ in
the next $n$ steps after a given visit is $\ge C/\log n$. Using the
strong Markov property one gets that the probability to make $k+1$
visits by time $n$ is $\le \exp(-Ck/\log n)$ and hence
%such that for any $k$, any $x$, and any $n\ge 2$,
%$$\E\left[\left(\frac{N_n(x)}{\ln n}\right)^k\right] \le C k!.$$
%As a consequence $N_n(x)/\ln n$ has finite small exponential
%moments. It follows, by using the Markov inequality, that
there exists some $C'>0$ depending on $M$ such that
$$\pp[N_n(x)\ge C'(\ln n)^2] =o( n^{-M-2}).$$
Moreover since $U$ makes nearest neighbor jumps, before time $n$ it stays in a ball of radius $n$. Thus if $N_n^*=\sup_x N_n(x)$, then
$$\pp[N_n^*\ge C'(\ln n)^2] =n^2\times o( n^{-M-2})=o( n^{-M}).$$
Thus if $r_{n,U}$ is the size of the range of $U$ at time $n$, we get
$$\pp\left[r_{n,U}\le n/(C'(\ln n)^2)\right] =o( n^{-M}).$$
Let's come back to the original process $S=(X,Y)$ now. We just observe that at time $n$ one of the $X$ or $Y$ component performed $n/2$ steps. Since each of these components is a simple random walk, we deduce from the previous estimate, that before time $n$,
$X$ or $Y$ will visit at least $n/(2C'(\ln n)^2)$ sites, with
probability at least $1-o(n^{-M})$. This gives the desired lower bound for $r_n$ and concludes the proof of the lemma.
\end{proof}

\noindent We can finish now the proof of Proposition \ref{prop}. As noticed in the introduction,
observe that the $X$ and $Y$ components are time changed simple random walks. Specifically we have the equality in law:
$$((X_k,Y_k),k\ge 0)=((U(r_k),V(k-r_k)), k\ge 0),$$
where $U$ and $V$ are two independent simple random walks on $\Z^2$ (and where by abuse of notation we also denote by $r_k$, the size of the range of the $(U,V)$ process at $k$-th' step). By using Lemma \ref{lem} and the independence of $U$ and $V$, we
get
\begin{eqnarray*}
\pp\left[0\in \{S_n,\dots,S_{2n}\}\right]& \le &\pp\left[0\in \{U(n/(C\ln n)^2),
\dots,U(2n)\}\right]\\
   && \times \pp\left[0\in \{V(n/100),
\dots,V(2n)\}\right] + o(n^{-M}).
\end{eqnarray*}
Thus Proposition \ref{prop} follows from the following lemma:
\begin{lem} Let $U$ be the simple random walk on $\Z^2$ and let $t\in [n/(\ln n)^3,2n]$. Then
\begin{eqnarray}
\label{returnU}
\pp\left[0\in \{U(t),
\dots,U(2n)\}\right] = \kO\left(\frac{\ln \ln n}{\ln n}\right).
\end{eqnarray}
\end{lem}
\begin{proof} This lemma is standard, but we give a proof for reader's convenience. First let $|\cdot|$ denotes some norm on $\R^2$.
Since $t\ge n/(\ln n)^3$, it is well known (see e.g. \cite[Theorem 2.1.1]{LL}) that
$$\pp\left[|U(t)|\le \frac{\sqrt n}{(\ln n)^3}\right] = \kO((\ln n)^{-1}).$$
Moreover for any $|x|\ge 4$ (see e.g. \cite[Proposition 1.6.7]{L}),
$$\pp_x[\tau_0<\tau_{|x|(\ln |x|)^4}] = \kO\left(\frac{\ln \ln |x|}{\ln |x|}\right),$$
where $\pp_x$ denotes the law of $U$ starting from $x$ and for any $r\ge 0$,
$$\tau_r= \inf\{k>0 \ :\ r<|U(k)|\le r+1\}.$$
But if $|x|\ge \sqrt{n}/(\ln n)^3$, then $n\ln n = \kO\left(|x|(\ln |x|)^4\right)$ and in particular (see e.g. \cite[Proposition 2.1.2]{LL})
$$\pp\left[\tau_{|x|(\ln |x|)^4} \le 2n \right] = \kO(n^{-1}).$$
Notice finally that if $|x|\ge \sqrt{n}/(\ln n)^3$, then
$$\frac{\ln \ln |x|}{\ln |x|} = \kO\left(\frac{\ln \ln n}{\ln n}\right).$$
The lemma follows by using the strong Markov property.
\end{proof}
\noindent The proof of Theorem \ref{theo} is now finished. \hfill $\square$

\begin{rem} \emph{The proof shows actually that for any finite initial configuration of visited sites, $M(2,2)$ is transient.
This is of course not always the case if this configuration is infinite. For instance if we decide that all sites of the form $(0,0,*,*)$ are already visited at time $0$, then the $X$ component will never move and $M(2,2)$ will not be transient.}
\end{rem}

\end{document}